\documentclass[a4,12pt]{amsart}
\usepackage{amssymb,amstext,amsmath,amscd,amsthm,amsfonts,enumerate,latexsym}
\usepackage{color}
\usepackage[dvipdfmx]{graphicx}
\usepackage[all]{xy}
\theoremstyle{plain}
\newtheorem{thm}{Theorem}[section]

\newtheorem*{thm*}{Theorem}
\newtheorem*{cor*}{Corollary}

\newtheorem{prop}[thm]{Proposition}

\newtheorem{lemma}[thm]{Lemma}

\newtheorem{cor}[thm]{Corollary}

\newtheorem*{claim*}{Claim}

\theoremstyle{definition}
\newtheorem{defn}[thm]{Definition}

\newtheorem{remark}[thm]{Remark}
\newtheorem{fact}[thm]{Fact}

\theoremstyle{remark}

\numberwithin{equation}{thm}

\def\Z{\mathbb{Z}}

\def\Min{\operatorname{Min}}

\def\Ext{\operatorname{Ext}}

\def\Hom{\operatorname{Hom}}

\def\e{\mathrm{e}}
\def\m{\mathfrak m}
\def\n{\mathfrak n}

\newcommand{\rme}{\mathrm{e}}

\newcommand{\rmr}{\mathrm{r}}

\newcommand{\rmv}{\mathrm{v}}

\newcommand{\rmK}{\mathrm{K}}

\newcommand{\fka}{\mathfrak{a}}

\newcommand{\fkm}{\mathfrak{m}}
\newcommand{\fkn}{\mathfrak{n}}

\newcommand{\fkp}{\mathfrak{p}}
\newcommand{\fkq}{\mathfrak{q}}

\newcommand{\fkM}{\mathfrak{M}}

\newcommand{\fkP}{\mathfrak{P}}

\newcommand{\mapright}[1]{%
\smash{\mathop{%
\hbox to 1cm{\rightarrowfill}}\limits^{#1}}}

\newcommand{\mapleft}[1]{%
\smash{\mathop{%
\hbox to 1cm{\leftarrowfill}}\limits_{#1}}}

\def\depth{\operatorname{depth}}
\def\AGL{\operatorname{AGL}}

\def\Ass{\operatorname{Ass}}

\def\id{\operatorname{id}}

\title{When is $R \ltimes I$ an almost Gorenstein local ring?}
\author{Shiro Goto}
\address{Department of Mathematics, School of Science and Technology, Meiji University, 1-1-1 Higashi-mita, Tama-ku, Kawasaki 214-8571, Japan}
\email{shirogoto@gmail.com}
\author{Shinya Kumashiro}
\address{Department of Mathematics and Informatics, Graduate School of Science and Technology, Chiba University, Chiba-shi 263, Japan}
\email{polar1412@gmail.com}

\thanks{The first author was partially supported by JSPS Grant-in-Aid for Scientific Research (C) 25400051.}

\begin{document}
\maketitle

\setlength{\baselineskip} {17pt}

\begin{abstract}
Let $(R, \m) $ be a Gorenstein local ring of dimension $d > 0$ and let $I$ be an ideal of $R$ such that $(0) \ne I \subsetneq R$ and $R/I$ is a Cohen-Macaulay ring of dimension $d$. There is given a complete answer to the question of when the idealization $A = R \ltimes I$ of $I$ over $R$ is an almost Gorenstein local ring.
\end{abstract}


\section{Introduction}\label{section1}
Let $(R,\m)$ be a Gorenstein local ring of dimension $d > 0$ with infinite residue class field. Assume that $R$ is a homomorphic image of a regular local ring. With this notation the purpose of this paper is to prove the following theorem.

\begin{thm}\label{maintheorem}
Let $I$ be a non-zero ideal of $R$ and suppose that $R/I$ is a Cohen-Macaulay ring of dimension $d$. Let $A = R\ltimes I$ denote the idealization of $I$ over $R$. Then the following conditions are equivalent.
\begin{enumerate}
\item[$(1)$] $A=R\ltimes I$ is an almost Gorenstein local ring. 
\item[$(2)$] $R$ has the presentation $R = S/[(X) \cap (Y)]$ where $S$ is a regular local ring of dimension $d +1$ and $X,Y$ is a part of a regular system of parameters of $S$ such that $I = XR$.

\end{enumerate}
\end{thm}

The notion of almost Gorenstein local ring ({\it AGL ring} for short) is one of the generalization of Gorenstein rings, which originated in the paper \cite{BF} of V. Barucci and R. Fr\"oberg in 1997. They introduced the notion for one-dimensional analytically unramified local rings and developed a beautiful theory, investigating the semigroup rings of numerical semigroups. In 2013 the first author, N. Matsuoka, and T. T. Phuong \cite{GMP} extended  the notion to arbitrary Cohen-Macaulay local rings but still of dimension one. The research of \cite{GMP} has been succeeded by two works \cite{GTT} and \cite{CGKM} in 2015 and 2017, respectively. In \cite{CGKM} one can find the notion of $2$-almost Gorenstein local ring ({\it $2$-AGL ring} for short) of dimension one, which is a generalization of AGL rings. Using the Sally modules of canonical ideals, the authors show that $2$-$\AGL$ rings behave well as if they were twins of AGL rings. The purpose of the research \cite{GTT} of the first author, R. Takahashi, and N, Taniguchi started in a different direction. They have extended the notion of AGL ring to higher dimensional Cohen-Macaulay local/graded rings, using the notion of Ulrich modules (\cite{BHU}). Here let us briefly recall their definition for the local case.

\begin{defn}\label{def}
Let $(R,\fkm)$ be a Cohen-Macaulay local ring of dimension $d$, possessing the canonical module $\rmK_R$. Then we say that $R$ is an $\AGL$ ring, if there exists an exact sequence
$$
0 \to R \to \rmK_R \to C \to 0
$$
of $R$-modules such that either $C = (0)$ or $C \ne (0)$ and $\mu_R(C) = \e^0_\fkm(C)$, where  $\mu_R(C)$ denotes the number of elements in a minimal system of generators of $C$ and $$\e^0_\fkm(C) = \lim_{n\to \infty}(d-1)!{\cdot}\frac{\ell_R(C/\fkm^{n+1}C)}{n^{d-1}}$$ denotes the multiplicity of $C$ with respect to the maximal ideal $\fkm$ (here $\ell_R(*)$ stands for the length).
\end{defn}

We explain a little about Definition \ref{def}. Let $(R,\fkm)$ be a Cohen-Macaulay local ring of dimension $d$ and assume that $R$ possesses the canonical module $\rmK_R$. The condition of Definition \ref{def} requires that  $R$ is embedded into   $\rmK_R$ and even though $R \ne \rmK_R$, the difference $C = \rmK_R/R$ between $\rmK_R$ and $R$ is an Ulrich $R$-module (\cite{BHU}) and behaves well. In particular, the condition is equivalent to saying that $\m C = (0)$, when $\dim R = 1$ (\cite[Proposition 3.4]{GTT}). In general, if $R$ is an AGL ring of dimension $d > 0$, then $R_\fkp$ is a Gorenstein ring for every $\fkp \in \Ass R$, because $\dim_RC \le d-1$ (\cite[Lemma 3.1]{GTT}).

The research on almost Gorenstein local/graded rings is still in progress, exploring, e.g., the problem of when the Rees algebras of ideals/modules are almost Gorenstein rings (see \cite{GMTY1, GMTY2, GMTY3, GMTY4, GRTT, Ta}) and the reader  can consult \cite{GTT} for several basic results on almost Gorenstein local/graded rings. For instance, non-Gorenstein AGL rings are G-regular in the sense of \cite{T} and all the known Cohen-Macaulay local rings of finite Cohen-Macaulay representation type are AGL rings. Besides, the authors explored the question of when the idealization $A = R \ltimes M$ is an AGL ring, where $(R,\m)$ is a Cohen-Macaulay local ring and $M$ is a maximal Cohen-Macaulay $R$-module. Because $A=R \ltimes M$ is a Gorenstein ring if and only if $M \cong \rmK_R$ as an $R$-module (\cite{R}), this question seems quite natural and  in \cite[Section 6]{GTT} the authors actually gave a complete answer to the question in the case where $M$ is a faithful $R$-module, that is the case $(0):_RM = (0)$. However, the case where $M$ is not faithful has been left  open, which our Theorem \ref{maintheorem} settles in the special case where $R$ is a Gorenstein local ring and $M  = I$ is an ideal of $R$ such that $R/I$ is a Cohen-Macaulay ring with $\dim R/I = \dim R$. For the case where $\dim R/I=d$ but $\depth R/I = d-1$ the question remains open (see Remark \ref{2.6}).

\section{Proof of Theorem \ref{maintheorem}}
The purpose of this section is to prove Theorem \ref{maintheorem}. To begin with, let us fix our notation. Unless otherwise specified, throughout this paper let $(R, \m) $ be a Gorenstein local ring with $d=\dim R >0$. Let $I$ be a non-zero ideal of $R$ such that $R/I$ is a Cohen-Macaulay ring with $\dim R/I = d$. Let $A = R \ltimes I$ be the idealization of $I$ over $R$.  Therefore, $A = R \oplus I$ as an $R$-module and the multiplication in $A$ is given  by
$$(a,x)(b,y) = (ab, bx + ay)$$
where $a,b \in R$ and $x,y \in I$. Hence $A$ is a Cohen-Macaulay local ring with $\dim A = d$, because $I$ is a maximal Cohen-Macaulay $R$-module.

For each $R$-module $N$ let $N^\vee = \Hom_R(N,R)$. We set $L = I^\vee  \oplus R$ and consider $L$ to be an $A$-module under the following action of $A$
$$(a,x)\circ (f,y) = (af, f(x) + ay),$$
where $(a,x) \in A$ and $(f,y) \in L$. Then it is standard to check that the map
$$A^\vee \to L, ~\alpha \mapsto (\alpha \circ j, \alpha (1))$$
is an isomorphism of $A$-modules, where $j : I \to A,~ x \mapsto (0,x)$ and $1 = (1,0)$ denotes the identity of the ring $A$. Hence by \cite[Satz 5.12]{HK} we get the following.

\begin{fact}\label{2.1}
$\rmK_A = L$, where $\rmK_A$ denotes the canonical module of $A$. 

\end{fact}

We set $J = (0):_RI$. Let $\iota : I \to R$ denote the embedding.  Then taking the $R$-dual of the exact sequence
$$0 \to I \xrightarrow{\iota} R \to R/I \to 0,$$ we get the exact sequence 
$$0 \to (R/I)^\vee \to R^\vee \xrightarrow{\iota^\vee} I^\vee \to 0 = \Ext_R^1(R/I,R)\to \cdots$$
of $R$-modules, which shows $I^\vee = R{\cdot} \iota$. Hence $J = (0):_RI^\vee$ because $I = I^{\vee \vee}$ (\cite[Korollar 6.8]{HK}), so that  $I^\vee =R{\cdot}\iota \cong R/J$ as an $R$-module. Hence $I \cong (R/J) ^\vee = \rmK_{R/J}$ (\cite[Satz 5.12]{HK}). Therefore, taking again the $R$-dual of the exact sequence
$$0 \to J \to R^\vee \xrightarrow{\iota^\vee}  I^\vee \to 0,$$ we get the exact sequence $0 \to I \xrightarrow{\iota}  R \to J^\vee \to 0$ of $R$-modules, whence $J^\vee \cong R/I$, so that $J \cong (R/I) ^\vee = \rmK_{R/I}$. Summarizing the arguments, we get the following.

\begin{fact}\label{2.2}
$I \cong (R/J) ^\vee = \rmK_{R/J}$ and $J \cong (R/I) ^\vee = \rmK_{R/I}$.
\end{fact}

Notice that  $\rmr (A) = 2$ by \cite[Satz 6.10]{HK} where $\rmr(A)$ denotes the Cohen-Macaulay type of $A$, because $A$ is not a Gorenstein ring (as $I \not\cong R$; see \cite{R}) but $\rmK_A$ is generated by two elements;
$\rmK_A = R{\cdot} (\iota,0) + R {\cdot} (0,1)$.

We denote by $\fkM=\m \times I$ the maximal ideal of $A$. Let us begin with the following.

\begin{lemma}\label{final}
Let $d=1$. Then the following conditions are equivalent.
\begin{enumerate} [$(1) $]
\item $A$ is an $\mathrm{AGL}$ ring.
\item $I +J = \fkm$.
\end{enumerate}
When this is the case, $I \cap J = (0)$.
\end{lemma}

\begin{proof} (2) $\Rightarrow$ (1) We set $f = (\iota,1) \in \rmK_A$ and $C = \rmK_A/Af$. Let $\alpha \in \fkm$ and $\beta \in I$. Let us write $\alpha = a + b$ with $a \in I$ and $b \in J$. Then because 
$$(\alpha, 0)(0,1)=(0,\alpha) =  (b\iota, a+b) = (b,a)(\iota,1), ~~(0,\beta)(0,1) = (0,0),$$
we get $\fkM C = (0)$, whence $A$ is an AGL ring.

 (1) $\Rightarrow$ (2) We have $I \cap J = (0)$. In fact, let $\fkp \in \Ass R$ and set $P = \fkp \times I$. Hence $P \in \Min A$. Assume that $IR_\fkp \ne (0)$. Then since $A_P = R_\fkp \ltimes IR_\fkp$ and $A_P$ is a Gorenstein local ring, $IR_\fkp \cong R_\fkp$ (\cite{R}), so that $JR_\fkp = (0)$. Therefore, $(I \cap J)R_\fkp = (0)$ for every $\fkp \in \Ass R$, whence $I \cap J = (0)$.

Now consider the exact sequence
$$0 \to A \xrightarrow{\varphi} \rmK_A \to C \to 0$$ 
of $A$-modules such that $\fkM C= (0)$. We set $f = \varphi (1)$. Then $f \not\in \fkM \rmK_A$ by \cite[Corollary 3.10]{GTT}, because $A$ is not a discrete valuation ring (DVR for short). We identify $\rmK_A = I^\vee \times R$ (Fact \ref{2.1}) and  write $f =(a\iota, b)$ with $a, b \in R$. Then $a \not\in \fkm $ or $b \not\in \fkm$, since $f = (a,0)(\iota,0) + (b,0)(0,1) \not\in \fkM \rmK_A$.

Firstly, assume that $a \not\in \fkm$. Without loss of generality, we may assume $a = 1$, whence $f = (\iota,b)$. Let $\alpha \in \fkm$. Then since $(\alpha,0)(0,1) \in Af$, we can write $(\alpha, 0)(0,1) = (r,x)(\iota,b)$ with some $r \in R$ and $x \in I$. Because $$(0, \alpha) = (\alpha, 0)(0,1)=(r,x)(\iota,b)=(r\iota, x + rb),$$ we get
$$r \in (0):_R\iota = J, ~~~\ \ \alpha = x + rb \in I + J.$$
Therefore, $\fkm = I+J$.

Now assume that $a \in \fkm$. Then since $b \not\in \fkm$, we may assume $b = 1$, whence $f = (a\iota, 1)$. Let $\alpha \in \fkm$ and write $(\alpha, 0)(\iota,0) = (r,x)(a\iota,1)$ with $r \in R$ and $x \in I$. Then since $(\alpha \iota,0)= \left((ra)\iota, ax + r\right)$, we get
$$\alpha - ra \in J, ~~~\ \  r = -xa \in (a), $$ 
so that $\alpha \in J +(a^2) \subseteq J + \fkm^2$, whence $\fkm = J$. Because $I \cap J = (0)$, this implies $I = (0)$, which is  absurd.  Therefore,  $a \not\in \fkm$, whence  $I + J = \fkm$.
\end{proof}

\begin{cor}\label{cor}
Let $d = 1$. Assume that $A = R \ltimes I$ is an $\mathrm{AGL}$ ring. Then both $R/I$ and $R/J$ are discrete valuation rings and $\mu_R(I) = \mu_R(J)=1$. Consequently, if $R$ is a homomorphic image of a regular local ring, then $R$ has the presentation $$R = S/[(X)\cap (Y)]$$ for some two-dimensional regular local ring $(S,\n)$ with $\n = (X,Y)$, so that $I=(x)$ and $J=(y)$, where $x$, $y$ respectively denote the images of $X$, $Y$ in $R$.
\end{cor}

\begin{proof}[Proof of Corollary $\ref{cor}$]
Since $I + J = \fkm$ and $I \cap J = (0)$, $\rmK_{R/I}  \cong J \cong \fkm/I$ by Fact \ref{2.2}. Hence $R/I$ is a DVR by Burch's Theorem (see, e.g., \cite[Theorem 1.1 (1)]{GH}), because $\id_{R/I}\fkm/I =\id_{R/I}\rmK_{R/I}= 1 < \infty$, where $\id_{R/I}(*)$ denotes the injective dimension. We similarly get that $R/J$ is a DVR, since $\rmK_{R/J} \cong I \cong {\fkm/J}$. Consequently, $\mu_R(I) = \mu_R(J)= 1$. We write $I = (x)$ and $J = (y)$. Hence $\fkm = I + J = (x,y)$. Since  $xy= 0$, we have $\fkm^2 = (x^2, y^2) = (x+y)\fkm$. Therefore, $\rmv(R) = \rme(R) = 2$ because $R$ is not a DVR, where $\rmv(R)$ (resp. $\rme(R)$) denotes the embedding dimension of $R$ (resp. the multiplicity $\rme_\m^0(R)$ of $R$ with respect to $\m$). Suppose now that $R$ is a homomorphic image of a regular local ring. Let us write $R = S/\fka$ where $\fka$ is an ideal in a  two-dimensional regular local ring $(S, \fkn)$  and  choose $X,Y \in \fkn$ so that $x, y$ are the images of $X,Y$ in $R$, respectively. Then $\fkn = (X,Y)$, since $\fka \subseteq \fkn^2$. We consider the canonical epimorphism $$\varphi: S/[(X) \cap (Y)] \to R$$ and get that $\varphi$ is  an isomorphism, because $$\ell_S \left(S/ (XY, X+Y)\right) = 2 = \ell_{R} \left(R/(x+y) R \right).$$ Thus $\fka = (X) \cap (Y)$ and $R = S/[(X) \cap (Y)]$.
\end{proof}

We note the following.

\begin{prop}\label{ex}
Let $S$ be a regular local ring of dimension $d + 1$~$(d > 0)$ and let $X,Y$ be a part of a regular system of parameters of $S$. We set $R = S/[(X) \cap (Y)]$ and $I=(x)$, where $x$ denotes the image of $X$ in $R$. Then $I \ne (0)$, $R/I$ is a Cohen-Macaulay ring with $\dim R/I = d$, and the idealization $A = R \ltimes I$ is an $\mathrm{AGL}$ ring.
\end{prop}

\begin{proof}
Let $y$ be the image of $Y$ in $R$. Then $(y) = (0):_Rx$ and we have the presentation 
$$0 \to (y) \to R \to (x) \to 0$$
of the $R$-module $I=(x)$, whence $A = R[T]/(yT, T^2)$, where $T$ is an indeterminate. Therefore
$$A = S[T]/(XY, YT, T^2).$$ Notice that $(XY, YT, T^2)$ is equal to the ideal generated by the $2 \times 2$ minors of the matrix $\Bbb M = \left(\begin{smallmatrix}
X&Y&T\\
T&Y&0\\
\end{smallmatrix}\right)$ and we readily get by \cite[Theorem 7.8]{GTT} that $A = R\ltimes I$ is an AGL ring, because $X,Y,T$ is a part of a regular system of parameters of the regular local ring  $S[T]_\fkP$, where $\fkP = \fkn S[T] + (T)$.
\end{proof}

We are now ready to prove Theorem \ref{maintheorem}.

\begin{proof}[Proof of Theorem $\ref{maintheorem}$]
By Proposition \ref{ex} we have only to show the implication $(1) \Rightarrow (2)$. Consider the exact sequence
$$0 \to A \to \rmK_A \to C \to 0$$
of $A$-modules such that $C$ is an Ulrich $A$-module. Let $\fkM = \fkm \times I$ stand for the maximal ideal of $A$. Then since $\fkm A \subseteq \fkM \subseteq \overline{\fkm A}$ (here $\overline{\fkm A}$ denotes the integral closure of $\fkm A$) and the field $R/\fkm$ is infinite, we can choose a superficial sequence $f_1, f_2, \ldots, f_{d-1} \in \fkm$ for $C$ with respect to $\fkM$ so that $f_1, f_2, \ldots, f_{d-1}$ is also a part of a system of parameters for both $R$ and $R/I$. We set $\fkq = (f_1, f_2, \ldots, f_{d-1})$ and $\overline{R} = R/\fkq$. Let $\overline{I} = (I + \fkq)/\fkq$ and $\overline{J} = (J+\fkq)/\fkq$. Then since $f_1, f_2, \ldots, f_{d-1}$ is a regular sequence for $R/I$,  by the exact sequence
$$0 \to I \to R \to R/I \to 0$$
we get the exact sequence
$$0 \to I/\fkq I \to \overline{R} \to R/(I + \fkq) \to 0,$$
so that $I/\fkq I \cong \overline{I}$ as an $\overline{R}$-module. Hence 
$$A/\fkq A = \overline{R} \ltimes (I/\fkq I) \cong \overline{R} \ltimes \overline{I}.$$ 
Remember that $A/\fkq A $ is an AGL ring by \cite[Theorem 3.7]{GTT}, because $f_1, f_2, \ldots, f_{d-1}$ is a superficial sequence of $C$ with respect $\fkM$ and $f_1, f_2, \ldots, f_{d-1}$ is an $A$-regular sequence. Consequently, thanks to Corollary \ref{cor}, $\overline{R}$ is a DVR and $\mu_{\overline{R}}(\overline{I}) = 1$. Hence $R/I$ is a regular local ring  and $\mu_R(I) = 1$, because $I/\fkq I \cong  \overline{I}$. Let $I = (x)$. Then $R/J \cong I=(x)$, since $J = (0):_RI$. Because $f_1, f_2, \ldots, f_{d-1}$ is a regular sequence for the $R$-module $I$, $f_1, f_2, \ldots, f_{d-1}$ is a regular sequence for $R/J$, so that we get the exact sequence
$$0 \to J/\fkq J \to \overline{R} \to R/(J + \fkq) \to 0.$$ Therefore, $\overline{J} \cong J/\fkq J$ and since $R/(J + \fkq) \cong I/\fkq I \cong \overline{I}$, we have $\overline{J} = (0):_{\overline{R}}\overline{I}$. Hence  $R/J$ is a regular local ring and $\mu_R(J)=1$, because $\overline{R}/\overline{J}$ is a DVR and $\mu_{\overline{R}}(\overline{J}) = 1$ by Corollary \ref{cor}.

Let $J = (y)$ and let $\overline{\fkm} = \fkm/\fkq$. Then by Lemma  \ref{final} we have $\overline{\fkm} = \overline{I} + \overline{J}$, whence $\fkm = (x,y, f_1, f_2, \ldots, f_{d-1})$. Therefore $\mu_R(\fkm) = d+1$, since $R$ is not a regular local ring. On the other hand, since both $R/I$ and $R/J$ are regular local rings, considering the canonical exact sequence
$$0 \to R \to R/I \oplus R/J \to R/(I+J) \to 0$$ (notice that $I \cap J=(0)$ for the same reason as in the proof of Lemma \ref{final}), we readily get $\rme(R) = 2$. We now choose a regular local ring  $(S,\fkn)$ of dimension $d+1$ and an ideal $\fka$ of $S$ so that $R = S/\fka$. Let $X, Y, \Z_1, Z_2, \ldots, Z_{d-1}$ be the elements of $\n$ whose images in $R$ are equal to $x, y, f_1, f_2, \ldots, f_{d-1}$, respectively. Then $\n = (X, Y, Z_1, Z_2, \ldots, Z_{d-1})$, since $\fka \subseteq \n^2$. Because $(X) \cap (Y) \subseteq \fka$ as  $xy = 0$, we get a surjective homomorphism $$ S/[(X) \cap (Y)] \to R$$ of rings, which has to be an isomorphism, because both the Cohen-Macaulay local rings $S/[(X) \cap (Y)]$ and $R$ have the same multiplicity $2$. This completes the proof of Theorem \ref{maintheorem}.
\end{proof}

\begin{remark}\label{2.6}
Let $(S, \n)$ be a two-dimensional regular local ring and let $X,Y$ be a regular system of parameters of $S$. We set $R = S/[(X) \cap (Y)]$. Let $x$, $y$ denote the images of $X$, $Y$ in $R$, respectively. Let $n \ge 2$ be an integer. Then $\dim R/(x^n)=1$ but $\depth R/(x^n) = 0$. We have $x^n=x^{n-1}(x+y)$, whence $(x^n) \cong (x)$ as an $R$-module because $x+y$ is a non-zerodivisor of $R$, so that $R\ltimes (x^n)$ is an $\AGL$ ring (Proposition \ref{ex}). This example shows that there are certain ideals $I$ in Gorenstein local rings $R$ of dimension $d > 0$ such that $\dim R/I = d$ and $\depth R/I = d-1$, for which the idealizations $R \ltimes I$ are AGL rings. However, we have no idea to control them.
\end{remark}



\end{document}